\documentclass[11pt]{article}
       \usepackage{amsfonts}
       \usepackage{stmaryrd}
       \usepackage{latexsym,amssymb,mathrsfs,fancyhdr}
       \font\tenmsb=msbm10
       \font\sevenmsb=msbm7
       \font\fivemsb=msbm5
       \catcode`\@=11
       \ifx\amstexloaded@\relax\catcode`\@=\active
       \endinput\else\let\amstexloaded@\relax\fi
       \def\spaces@{\space\space\space\space\space}
       \def\spaces@@{\spaces@\spaces@\spaces@\spaces@\spaces@}
       \def\space@.  {\futurelet\space@\relax}
       \space@.   %
       \def\Err@#1{\errhelp\defaulthelp@\errmessage{AmS-TeX error: #1}}
       \def\relaxnext@{\let\next\relax}
       \def\accentfam@{7}
       \def\noaccents@{\def\accentfam@{0}}
       \def\Cal{\relaxnext@\ifmmode\let\next\Cal@\else
       \def\next{\Err@{Use \string\Cal\space only in math mode}}\fi\next}
       \def\Cal@#1{{\Cal@@{#1}}}
       \def\Cal@@#1{\noaccents@\fam\tw@#1}
       \def\Bbb{\relaxnext@\ifmmode\let\next\Bbb@\else
       \def\next{\Err@{Use \string\Bbb\space only in math mode}}\fi\next}
       \def\Bbb@#1{{\Bbb@@{#1}}}
       \def\Bbb@@#1{\noaccents@\fam\msbfam#1}
       \newfam\msbfam
       \textfont\msbfam=\tenmsb
       \scriptfont\msbfam=\sevenmsb
       \scriptscriptfont\msbfam=\fivemsb
\usepackage{dsfont}
\usepackage[german,english]{babel}
\usepackage{amsmath,amssymb}
\usepackage[square, comma, sort&compress, numbers]{natbib}
\usepackage{cases}
\usepackage{mathrsfs}
\usepackage{amsmath}
\usepackage{amsthm}
\usepackage{amsfonts}
\usepackage{amssymb}
\usepackage{latexsym}
\usepackage{fancyhdr}
\usepackage{extarrows}

\newtheorem{thm}{Theorem}[section]
\newtheorem{prop}[thm]{Proposition}
\newtheorem{lem}[thm]{Lemma}
\newtheorem{rem}[thm]{Remark}
\newtheorem{iteration lemma}[thm]{iteration Lemma}

\newtheorem*{acknowledgements*}{ACKNOWLEDGEMENTS}

\begin{document}

\setlength{\columnsep}{5pt}
\title{\bf Core partial order in rings with involution}
\author{Xiaoxiang  Zhang\footnote{ E-mail: z990303@seu.edu.cn},
\ Sanzhang  Xu\footnote{ E-mail: xusanzhang5222@126.com},
\ Jianlong Chen\footnote{ Corresponding author. E-mail: jlchen@seu.edu.cn }.\\
Department of  Mathematics, Southeast University \\  Nanjing 210096,  China }
     \date{}

\maketitle
\begin{quote}
{\textbf{Abstract:} \small
  Let $R$ be a unital ring with involution.
  Several characterizations and properties of core partial order are given.
  In particular, we investigate the reverse order law $(ab)^{\tiny\textcircled{\tiny\#}}=b^{\tiny\textcircled{\tiny\#}}a^{\tiny\textcircled{\tiny\#}}$
  for two core invertible elements $a,b\in R$.
  Some relationships between core partial order and other partial orders are obtained.


}
\end{quote}

\section{ Introduction }\label{a}
The core inverse of a complex matrix was introduced by Baksalary and Trenkler \cite{BT}.
Let $M_{n}(\mathbb{C})$ be the ring of all $n\times n$ complex matrices. A matrix $X\in M_{n}(\mathbb{C})$ is called a core inverse of $A\in M_{n}(\mathbb{C})$, if it satisfies
$AX=P_{A}$ and $\mathcal{R}(X)\subseteq \mathcal{R}(A)$,
where $\mathcal{R}(A)$ denotes the column space of $A$,
and $P_{A}$ is the orthogonal projector onto $\mathcal{R}(A)$.
And if such a matrix $X$ exists, then it is unique and denoted by $A^{\tiny\textcircled{\tiny\#}}$.
The core partial order for a complex matrix were also introduced in \cite{BT}.
Let $\mathbb{C}^{CM}_{n}=\{A\in M_{n}(\mathbb{C})\mid \mathrm{rank}(A)=\mathrm{rank}(A^{2})\}$
, $A\in \mathbb{C}^{CM}_{n}$ and $B\in M_{n}(\mathbb{C})$. The binary operation $\overset{\tiny\textcircled{\tiny\#}}\leq$ is defined as follows:
$$A\overset{\tiny\textcircled{\tiny\#}}\leq B~\Leftrightarrow~
A^{\tiny\textcircled{\tiny\#}}A=A^{\tiny\textcircled{\tiny\#}}B~~\mathrm{and}~~AA^{\tiny\textcircled{\tiny\#}}=BA^{\tiny\textcircled{\tiny\#}}.$$
In \cite[Theorem 6]{BT}, it is proved that core partial order is a matrix partial order. Baksalary and Trenkler gave several
characterizations and various relationships between the matrix core partial order and other matrix partial orders by using the decomposition of Hartwig and Spindelb\"{o}ck \cite{HS}.
In \cite{RD}, Raki\'{c} and Djordjevi\'{c} generalized the matrix core partial order to the ring case. They gave various equivalent conditions of core partial
order and investigated relationships between the core partial order and other partial orders in general rings.
Motivated by \cite{BT,M2,MRT,R,RD}, in this paper, we give some new equivalent conditions and properties for core partial order in general rings.
Moreover, some new relationships between core partial order and other partial orders are obtained. As an application, we prove the reverse law for two core invertible
elements under the core partial order.

Let $R$ be a $*$-ring, that is a ring with an involution $a\mapsto a^*$
satisfying $(a^*)^*=a$, $(ab)^*=b^*a^*$ and $(a+b)^*=a^*+b^*$ for all $a,b\in R$.
We say that $x\in R$ is the Moore-Penrose inverse of $a\in R$, if the following hold:
$$axa=a, \quad xax=x, \quad (ax)^{\ast}=ax \quad (xa)^{\ast}=xa.$$
There is at most one $x$ such that above four equations hold.
If such an element $x$ exists, it is denoted by $a^{\dagger}$. The set of all Moore-Penrose invertible elements will be denoted by $R^{\dagger}$.
An element $x\in R$ is an inner inverse of $a\in R$ if $axa=a$ holds. The set of all inner inverses of $a$ will be denoted by $a\{1\}$.
An element $a\in R$ is said to be group invertible
if there exists $x\in R$ such that the following equations hold:
$$axa=a, \quad xax=x, \quad ax=xa.$$
The element $x$ which satisfies the above equations is called a group inverse of $a$.
If such an element $x$ exists, it is unique and denoted by $a^\#$. The set of all group invertible elements will be denoted by $R^\#$.
An element $a\in R$ is said to be an EP element if $a\in R^{\dagger}\cap R^\#$ and $a^{\dagger}=a^\#.$ The set of all EP elements will be denoted by $R^{EP}$.
In \cite{RDD} Raki\'{c}, Din\v{c}i\'{c} and Djordjevi\'{c} generalized the core inverse of a complex matrix to the case of an element in a ring.
Let $a,x\in R$, if
$$axa=a,~xR=aR,~Rx=Ra^{\ast},$$
then $x$ is called a core inverse of $a$ and if such an element $x$ exists, then it is unique and denoted by $a^{\tiny{\textcircled{\tiny\#}}}$. The set of all core invertible elements in $R$ will be denoted by $R^{\tiny{\textcircled{\tiny\#}}}$.
An element $p\in R$ is called self-adjoint idempotent if $p^{2}=p=p^{\ast}$.
An element $q\in R$ is called idempotent if $q^{2}=q$.

For $a,b\in R$, we have the following definitions:
\begin{itemize}
\item[{\rm $\bullet$}] the star partial order $a\overset{\ast}\leq b$: $a^{\ast}a=a^{\ast}b$ and $aa^{\ast}=ba^{\ast}$\cite{D};
\item[{\rm $\bullet$}] the minus partial order $a\overset{-}\leq b$ if and only if there exists an $a^{-}\in a\{1\}$ such that $a^{-}a=a^{-}b$ and $aa^{-}=ba^{-}$\cite{H2};
\item[{\rm $\bullet$}] the sharp partial order $a\overset{\#}\leq b$: $a^{\#}a=a^\#b$ and $aa^{\#}=ba^{\#}$\cite{M}.
\end{itemize}
This paper is organized as follows. In section 2, some new equivalent characterizations of the core partial order in rings are obtained.
Specially, the reverse order of two core invertible elements in rings was given. In section 3, some relationships of the core partial order and other
partial orders are obtained.

\section{ Equivalent conditions and properties of core partial order }\label{a}
In this section, some new characterizations of the core partial order in rings are obtained. Let us start this section with two auxiliary lemmas.
These two lemmas can be found in \cite[Lemma 2.2]{M} and \cite[Lemma 2.3 and Theorem 2.6]{RD}.
\begin{lem} \label{lemma-partial1}
Let $a\in R^\#$ and $b\in R$. Then:
\begin{itemize}
\item[{\rm (1)}] $a^\#a=a^\#b$ if and only if $a^{2}=ab$;
\item[{\rm (2)}] $aa^\#=ba^\#$ if and only if $a^{2}=ba$;
\item[{\rm (3)}] $a\overset{\#}\leq b$ if and only if $a^{2}=ab=ba$;
\item[{\rm (4)}] $a\overset{\#}\leq b$ if and only if there exists idempotent $p\in R$ such that $a=pb=bp$.
\end{itemize}
\end{lem}

\begin{lem} \label{lemma-partial2}
Let $a\in R^{\tiny{\textcircled{\tiny\#}}}$ and $b\in R$. Then:
\begin{itemize}
\item[{\rm (1)}] $a^{\tiny{\textcircled{\tiny\#}}} a=a^{\tiny{\textcircled{\tiny\#}}}b$ if and only if $a^{*}a=a^{*}b$;
\item[{\rm (2)}] $aa^{\tiny{\textcircled{\tiny\#}}}=ba^{\tiny{\textcircled{\tiny\#}}}$ if and only if $a^{2}=ba$ if and only if $aa^{\#}=ba^{\#}$.
\end{itemize}
\end{lem}

We will use the following notations $aR=\{ax\mid x\in R\}$, $Ra=\{xa\mid x\in R\}$, $^{\circ}a=\{x\in R\mid xa=0\}$ and $a^{\circ}=\{x\in R\mid ax=0\}$.

In \cite[Lemma 8]{LPT},  Lebtahi et al. proved that $a\overset{-}\leq b$ if and only if there exists
$c\in a\{1,2\}$ such that $b-a\in$ $^{\circ}c\cap c^{\circ}$.
For the core partial order, we have the following result.

\begin{thm} \label{core-pr1}
Let $a\in R^{\tiny{\textcircled{\tiny\#}}}$ and $b\in R$. Then the following conditions are equivalent:
\begin{itemize}
\item[{\rm (1)}] $a\overset{\tiny{\textcircled{\tiny\#}}}\leq b$;
\item[{\rm (2)}] $ba^{\tiny{\textcircled{\tiny\#}}}b=a$ and $a^{\tiny{\textcircled{\tiny\#}}}ba^{\tiny{\textcircled{\tiny\#}}}=a^{\tiny{\textcircled{\tiny\#}}}$;
\item[{\rm (3)}] $aa^{\tiny{\textcircled{\tiny\#}}}b=a=ba^{\tiny{\textcircled{\tiny\#}}}a$;
\item[{\rm (4)}] $b-a\in ^{\circ}\!\!a\cap (a^{\ast})^{\circ}$;
\item[{\rm (5)}] $b-a\in(1-aa^{\tiny{\textcircled{\tiny\#}}})R\cap R(1-aa^{\tiny{\textcircled{\tiny\#}}})$;
\item[{\rm (6)}] $b-a\in ^{\circ}\!\!(aa^{\tiny{\textcircled{\tiny\#}}}) \cap (aa^{\tiny{\textcircled{\tiny\#}}})^{\circ}$.
\end{itemize}
\end{thm}
\begin{proof}
$(1)\Leftrightarrow(2)$ Suppose that $a\overset{\tiny{\textcircled{\tiny\#}}}\leq b$. Then
$ba^{\tiny{\textcircled{\tiny\#}}}b
       =aa^{\tiny{\textcircled{\tiny\#}}}b
       =aa^{\tiny{\textcircled{\tiny\#}}}a=a$ and
        $a^{\tiny{\textcircled{\tiny\#}}}ba^{\tiny{\textcircled{\tiny\#}}}
       =a^{\tiny{\textcircled{\tiny\#}}}aa^{\tiny{\textcircled{\tiny\#}}}
       =a^{\tiny{\textcircled{\tiny\#}}}.$
Conversely, if $ba^{\tiny{\textcircled{\tiny\#}}}b=a$ and
$a^{\tiny{\textcircled{\tiny\#}}}ba^{\tiny{\textcircled{\tiny\#}}}=a^{\tiny{\textcircled{\tiny\#}}}$, then $aa^{\tiny{\textcircled{\tiny\#}}}
          =ba^{\tiny{\textcircled{\tiny\#}}}ba^{\tiny{\textcircled{\tiny\#}}}
          =ba^{\tiny{\textcircled{\tiny\#}}}$
          and
          $a^{\tiny{\textcircled{\tiny\#}}}a
          =a^{\tiny{\textcircled{\tiny\#}}}ba^{\tiny{\textcircled{\tiny\#}}}b
          =a^{\tiny{\textcircled{\tiny\#}}}b.$

$(1)\Leftrightarrow(3)$ Suppose that $a\overset{\tiny{\textcircled{\tiny\#}}}\leq b$. Then
$a^{\tiny{\textcircled{\tiny\#}}}a=a^{\tiny{\textcircled{\tiny\#}}}b$ and
$aa^{\tiny{\textcircled{\tiny\#}}}=ba^{\tiny{\textcircled{\tiny\#}}}$. Thus
$aa^{\tiny{\textcircled{\tiny\#}}}b=aa^{\tiny{\textcircled{\tiny\#}}}a=a$
and $ba^{\tiny{\textcircled{\tiny\#}}}a=aa^{\tiny{\textcircled{\tiny\#}}}a=a$.
Conversely, if $aa^{\tiny{\textcircled{\tiny\#}}}b=a=ba^{\tiny{\textcircled{\tiny\#}}}a$,
then pre-multiplication by $a^{\tiny{\textcircled{\tiny\#}}}$ on $aa^{\tiny{\textcircled{\tiny\#}}}b=a$ yields
$a^{\tiny{\textcircled{\tiny\#}}}b=a^{\tiny{\textcircled{\tiny\#}}}a$, similarly we have
$ba^{\tiny{\textcircled{\tiny\#}}}=aa^{\tiny{\textcircled{\tiny\#}}}$, thus
$a\overset{\tiny{\textcircled{\tiny\#}}}\leq b$.

$(1)\Leftrightarrow(4)$ Since $b-a\in ^{\circ}\!\!a\cap (a^{\ast})^{\circ}$ is equivalent to both $a^{*}a=a^{*}b$ and $a^{2}=ba$ hold,
thus $(1)\Leftrightarrow(4)$ by Lemma \ref{lemma-partial2}.

$(4)\Leftrightarrow(5)$ By $a\in R^{\tiny{\textcircled{\tiny\#}}}$, we have $^{\circ}a=R(1-aa^{\tiny{\textcircled{\tiny\#}}})$ and
 $(a^{\ast})^{\circ}=(1-(a^{\tiny{\textcircled{\tiny\#}}})^{\ast}a^{\ast})R
                    =(1-(aa^{\tiny{\textcircled{\tiny\#}}})^{\ast})R
                    =(1-aa^{\tiny{\textcircled{\tiny\#}}})R.$

$(5)\Leftrightarrow(6)$ By $(aa^{\tiny{\textcircled{\tiny\#}}})^{2}=aa^{\tiny{\textcircled{\tiny\#}}}$, we have 
$(1-aa^{\tiny{\textcircled{\tiny\#}}})R=(aa^{\tiny{\textcircled{\tiny\#}}})^{\circ}$ and
$R(1-aa^{\tiny{\textcircled{\tiny\#}}})=^{\circ}\!(aa^{\tiny{\textcircled{\tiny\#}}}).$
\end{proof}

If $p,~q\in R$ are idempotents, then arbitrary $a\in R$ can be written as
$$a=paq+pa(1-q)+(1-p)aq+(1-p)a(1-q).$$
The corresponding matrix form is
$$a= \left[
            \begin{matrix}
              a_{11} & a_{12} \\
              a_{21} & a_{22}
            \end{matrix}
         \right]_{p\times q},$$
where $a_{11}=paq$, $a_{12}=pa(1-q)$, $a_{21}=(1-p)aq$ and $a_{22}=(1-p)a(1-q)$. If $a=(a_{ij})_{p\times q}$ and $b=(b_{ij})_{p\times q}$, then $a+b=(a_{ij}+b_{ij})_{p\times q}$.

In \cite[Theorem 2.6]{RD}, Raki\'{c} and Djordjevi\'{c} proved that $a\overset{\tiny{\textcircled{\tiny\#}}}\leq b$ if and only if
there exist self-adjoint idempotent $p\in R$ and idempotent $q\in R$ such that $a=pb=bq$ and $qa=a$.
We now provide some new characterizations for the core partial order in terms of self-adjoint idempotents.

\begin{thm} \label{core-pr2}
Let $a\in R^{\tiny{\textcircled{\tiny\#}}}$ and $b\in R$. Then the following conditions are equivalent:
\begin{itemize}
\item[{\rm (1)}] $a\overset{\tiny{\textcircled{\tiny\#}}}\leq b$;
\item[{\rm (2)}] There exists a self-adjoint idempotent $p\in R$ such that $a=pb$, $ap=bp$ and $aR=pR$;
\item[{\rm (3)}] There exists self-adjoint idempotent $p\in R$ such that $a=pb$, $ap=bp$;
\item[{\rm (4)}] $a= \left(
            \begin{smallmatrix}
              a_{1} & a_{2} \\
                    &       \\
              0 & 0
            \end{smallmatrix}
         \right)_{p\times p},~~
        b= \left(
            \begin{smallmatrix}
              a_{1} & a_{2} \\
                    &       \\
               0    & b_{4}
            \end{smallmatrix}
         \right)_{p\times p}$.
\end{itemize}
\end{thm}
\begin{proof}
$(1)\Rightarrow(2)$ Let $p=aa^{\tiny{\textcircled{\tiny\#}}}$, then $p^{2}=p=p^{\ast}$ and
    $pb=aa^{\tiny{\textcircled{\tiny\#}}}b=aa^{\tiny{\textcircled{\tiny\#}}}a=a$,
       $ap
       =a^{2}a^{\tiny{\textcircled{\tiny\#}}}=aa^{\tiny{\textcircled{\tiny\#}}}a^{2}a^{\tiny{\textcircled{\tiny\#}}}
       =ba^{\tiny{\textcircled{\tiny\#}}}a^{2}a^{\tiny{\textcircled{\tiny\#}}}
       =baa^{\tiny{\textcircled{\tiny\#}}}
       =bp$,
$aR=pR$ by $a=aa^{\tiny{\textcircled{\tiny\#}}}a=pa$.

$(2)\Rightarrow(3)$ It is trivial.

$(3)\Rightarrow(1)$ Suppose that $a=pb$ and $ap=bp$. Then $a^{2}=apb=bpb=ba$
and $a^{*}a=(pb)^{*}pb=b^{*}p^{*}pb=b^{*}p^{\ast}b=(pb)^{*}b=a^{\ast}b$, thus $a\overset{\tiny{\textcircled{\tiny\#}}}\leq b$ by Lemma \ref{lemma-partial2}.

$(3)\Rightarrow(4)$ Suppose that $a=pb$ and $ap=bp$. Then $pa=a$ and
$$\begin{array}{rcl}
 pap=ap=a_{1},       &~~~~ & pa(1-p)=a-ap=a_{2},\\
(1-p)ap=0,           &~~~~ & (1-p)a(1-p)=0.\\
 pbp=ap=a_{1},       &~~~~ & pb(1-p)=a-ap=a_{2},\\
(1-p)bp=ap-ap=0,     &~~~~ & (1-p)b(1-p)=b-a=b_{4}.
\end{array}$$
Thus $a= \left(
            \begin{smallmatrix}
              a_{1} & a_{2}\\
                    &      \\
              0 & 0
            \end{smallmatrix}
         \right)_{p\times p},~~
        b= \left(
            \begin{smallmatrix}
              a_{1} & a_{2} \\
                    &       \\
               0    & b_{4}
            \end{smallmatrix}
         \right)_{p\times p}$.\\
$(4)\Rightarrow(3)$  If there exists $p^{2}=p=p^{\ast}$ such that $pa=a$, $a_{1}=ap$, $a_{2}=a-ap$, $b_{4}=b-a$, then
    $$pb=\left(
            \begin{smallmatrix}
              p & 0\\
                    &      \\
              0 & 0
            \end{smallmatrix}
         \right)_{p\times p}
     \left(
            \begin{smallmatrix}
              a_{1} & a_{2} \\
                    &      \\
               0    & b_{4}
            \end{smallmatrix}
         \right)_{p\times p}
   =\left(
            \begin{smallmatrix}
              pa_{1} & pa_{2} \\
                    &      \\
               0    & 0
            \end{smallmatrix}
         \right)_{p\times p}
   =\left(
            \begin{smallmatrix}
              a_{1} & a_{2} \\
                    &      \\
               0    & 0
            \end{smallmatrix}
         \right)_{p\times p}
   =a,$$

$$ap=\left(
            \begin{smallmatrix}
              a_{1} & a_{2} \\
                    &      \\
               0    & 0
            \end{smallmatrix}
         \right)_{p\times p}
  \left(
            \begin{smallmatrix}
                p   &   0 \\
                    &      \\
                0   &   0
            \end{smallmatrix}
         \right)_{p\times p}
 =\left(
            \begin{smallmatrix}
              a_{1}p & 0 \\
                    &      \\
               0    &  0
            \end{smallmatrix}
         \right)_{p\times p}
 =\left(
            \begin{smallmatrix}
              a_{1} &  0 \\
                    &      \\
               0    &  0
            \end{smallmatrix}
         \right)_{p\times p},$$
$$bp=\left(
            \begin{smallmatrix}
              a_{1} & a_{2} \\
                    &      \\
               0    & b_{4}
            \end{smallmatrix}
         \right)_{p\times p}
    \left(
            \begin{smallmatrix}
               p    & 0 \\
                    &      \\
               0    & 0
            \end{smallmatrix}
         \right)_{p\times p}
  =\left(
            \begin{smallmatrix}
              a_{1}p &  0\\
                    &      \\
               0    &   0
            \end{smallmatrix}
         \right)_{p\times p}
 =\left(
            \begin{smallmatrix}
              a_{1} &  0 \\
                    &      \\
               0    &  0
            \end{smallmatrix}
         \right)_{p\times p}.$$
Hence, $pb=a,~ ap=bp$.
\end{proof}

The following characterizations of the minus partial order will be used in the proof of Theorem 2.6,
which plays an important role in the sequel.
\begin{lem} \emph{\cite[Lemma 3.4]{Rb}} \label{minus-regular}
Let $a,~b\in R^{-}$. The following conditions are equivalent :
\begin{itemize}
\item[{\rm (1)}] $a\overset{-}\leq b$;
\item[{\rm (2)}] There exists $b^{-}\in b\{1\}$ such that $a=bb^{-}a=ab^{-}b=ab^{-}a$;
\item[{\rm (3)}] For arbitrary $b^{-}\in b\{1\}$, we have $a=bb^{-}a=ab^{-}b=ab^{-}a$.
\end{itemize}
\end{lem}

\begin{thm} \label{core-minus-regular}
Let $a,~b \in R^{\tiny{\textcircled{\tiny\#}}}$ with $a\overset{\tiny{\textcircled{\tiny\#}}}\leq b$. Then:
\begin{itemize}
\item[{\rm (1)}] $ba^{\tiny{\textcircled{\tiny\#}}}=ab^{\tiny{\textcircled{\tiny\#}}}$,
                 $a^{\tiny{\textcircled{\tiny\#}}}b=b^{\tiny{\textcircled{\tiny\#}}}a$;
\item[{\rm (2)}] $b^{\tiny{\textcircled{\tiny\#}}}ba^{\tiny{\textcircled{\tiny\#}}}
                 =a^{\tiny{\textcircled{\tiny\#}}}bb^{\tiny{\textcircled{\tiny\#}}}
                 =a^{\tiny{\textcircled{\tiny\#}}}ba^{\tiny{\textcircled{\tiny\#}}}
                 =a^{\tiny{\textcircled{\tiny\#}}}$;
\item[{\rm (3)}] $b^{\tiny{\textcircled{\tiny\#}}}aa^{\tiny{\textcircled{\tiny\#}}}
                 =a^{\tiny{\textcircled{\tiny\#}}}ab^{\tiny{\textcircled{\tiny\#}}}
                 =b^{\tiny{\textcircled{\tiny\#}}}ab^{\tiny{\textcircled{\tiny\#}}}
                 =a^{\tiny{\textcircled{\tiny\#}}}$.
\end{itemize}
\end{thm}
\begin{proof}
Suppose $a\overset{\tiny{\textcircled{\tiny\#}}}\leq b$, thus $a\overset{-}\leq b$ by $a^{\tiny{\textcircled{\tiny\#}}}\in a\{1\}$,
then $a=bb^{\tiny{\textcircled{\tiny\#}}}a=bb^{\#}a$ by Lemma \ref{minus-regular}.\\
$(1)$ $ba^{\tiny{\textcircled{\tiny\#}}}
           =aa^{\tiny{\textcircled{\tiny\#}}}=bb^{\tiny{\textcircled{\tiny\#}}}aa^{\tiny{\textcircled{\tiny\#}}}
           =(bb^{\tiny{\textcircled{\tiny\#}}}aa^{\tiny{\textcircled{\tiny\#}}})^{*}
           =aa^{\tiny{\textcircled{\tiny\#}}}bb^{\tiny{\textcircled{\tiny\#}}}
           =aa^{\tiny{\textcircled{\tiny\#}}}ab^{\tiny{\textcircled{\tiny\#}}}
           =ab^{\tiny{\textcircled{\tiny\#}}}$.

       $a^{\tiny{\textcircled{\tiny\#}}}b
        =b^{\tiny{\textcircled{\tiny\#}}}ab^{\tiny{\textcircled{\tiny\#}}}b
        =b^{\tiny{\textcircled{\tiny\#}}}ba^{\tiny{\textcircled{\tiny\#}}}b
        =b^{\tiny{\textcircled{\tiny\#}}}aa^{\tiny{\textcircled{\tiny\#}}}b
        =b^{\tiny{\textcircled{\tiny\#}}}aa^{\tiny{\textcircled{\tiny\#}}}a
        =b^{\tiny{\textcircled{\tiny\#}}}a.$\\
$(2)$ It is obviously $b^{\tiny{\textcircled{\tiny\#}}}ba^{\tiny{\textcircled{\tiny\#}}}
     =b^{\tiny{\textcircled{\tiny\#}}}ab^{\tiny{\textcircled{\tiny\#}}}
     =a^{\tiny{\textcircled{\tiny\#}}}$
     ,
     $a^{\tiny{\textcircled{\tiny\#}}}bb^{\tiny{\textcircled{\tiny\#}}}
     =b^{\tiny{\textcircled{\tiny\#}}}ab^{\tiny{\textcircled{\tiny\#}}}
     =a^{\tiny{\textcircled{\tiny\#}}}$
     and
   $a^{\tiny{\textcircled{\tiny\#}}}ba^{\tiny{\textcircled{\tiny\#}}}
   =a^{\tiny{\textcircled{\tiny\#}}}aa^{\tiny{\textcircled{\tiny\#}}}
   =a^{\tiny{\textcircled{\tiny\#}}}.$\\
$(3)$ Similarly to $(2)$, we have
   $b^{\tiny{\textcircled{\tiny\#}}}aa^{\tiny{\textcircled{\tiny\#}}}
   =b^{\tiny{\textcircled{\tiny\#}}}ba^{\tiny{\textcircled{\tiny\#}}}=a^{\tiny{\textcircled{\tiny\#}}}$
,
$a^{\tiny{\textcircled{\tiny\#}}}ab^{\tiny{\textcircled{\tiny\#}}}
=a^{\tiny{\textcircled{\tiny\#}}}bb^{\tiny{\textcircled{\tiny\#}}}=a^{\tiny{\textcircled{\tiny\#}}}$
and
$b^{\tiny{\textcircled{\tiny\#}}}ab^{\tiny{\textcircled{\tiny\#}}}
=b^{\tiny{\textcircled{\tiny\#}}}ba^{\tiny{\textcircled{\tiny\#}}}=a^{\tiny{\textcircled{\tiny\#}}}.$
\end{proof}

A complex matrix $A\in M_{n}(\mathbb{C})$ is called range-Hermite (EP matrix), if $\mathcal{R}(A)=\mathcal{R}(A^{\ast})$.
\begin{rem}\label{slip1}
\emph{In \cite[Theorem $2.4$]{M2}}, it is claimed that the following are equivalent for two complex matrices $A,B$ of index $1$ with the same order:
\begin{itemize}
\item[{\rm (1)}] $A^{\tiny{\textcircled{\tiny\#}}}BA^{\tiny{\textcircled{\tiny\#}}}=A^{\tiny{\textcircled{\tiny\#}}}$;
\item[{\rm (2)}] $A^{\dagger}BA^\#=A^{\tiny{\textcircled{\tiny\#}}}.$
\end{itemize}
\emph{While the implication $(2)\Rightarrow(1)$ is always valid,
the converse is not true in genral. In fact,
let
$A=B=\left[\begin{matrix}
              1 & 1 \\
              0 & 0
       \end{matrix}
       \right]\in M_{2}(\mathbb{C})$\normalsize,
we have $A^\#=A$, $A^{\dagger}=\left[\begin{matrix}
              1/2 & 0 \\
              1/2 & 0
       \end{matrix}
       \right]$\normalsize
and $A^{\tiny{\textcircled{\tiny\#}}}=\left[\begin{matrix}
              1 & 0 \\
              0 & 0
       \end{matrix}
       \right]$\normalsize,
then the condition
$A^{\tiny{\textcircled{\tiny\#}}}BA^{\tiny{\textcircled{\tiny\#}}}
=A^{\tiny{\textcircled{\tiny\#}}}AA^{\tiny{\textcircled{\tiny\#}}}=A^{\tiny{\textcircled{\tiny\#}}}$ holds.
However, $A^{\dagger}BA^\#\neq A^{\tiny{\textcircled{\tiny\#}}}.$
Note that $(1)\Rightarrow(2)$ holds in case $A$ is an EP matrix.}
\end{rem}

\begin{prop}  \label{a-b-core}
Let $a,b \in R^{\tiny{\textcircled{\tiny\#}}}$. Then $a\overset{\tiny{\textcircled{\tiny\#}}}\leq b$ if and only if
$a^{\tiny{\textcircled{\tiny\#}}}b=b^{\tiny{\textcircled{\tiny\#}}}a$,
$ba^{\tiny{\textcircled{\tiny\#}}}=ab^{\tiny{\textcircled{\tiny\#}}}$,
$ab^{\tiny{\textcircled{\tiny\#}}}a=a.$
\end{prop}
\begin{proof}
Suppose that $a\overset{\tiny{\textcircled{\tiny\#}}}\leq b$. Then
     $a^{\tiny{\textcircled{\tiny\#}}}b=b^{\tiny{\textcircled{\tiny\#}}}a$ and
     $ba^{\tiny{\textcircled{\tiny\#}}}=ab^{\tiny{\textcircled{\tiny\#}}}$
by Theorem \ref{core-minus-regular}, thus
$ab^{\tiny{\textcircled{\tiny\#}}}a=ba^{\tiny{\textcircled{\tiny\#}}}a=aa^{\tiny{\textcircled{\tiny\#}}}a=a.$
Conversely, if $a^{\tiny{\textcircled{\tiny\#}}}b=b^{\tiny{\textcircled{\tiny\#}}}a$,
$ba^{\tiny{\textcircled{\tiny\#}}}=ab^{\tiny{\textcircled{\tiny\#}}}$,
$ab^{\tiny{\textcircled{\tiny\#}}}a=a$, then
       $a^{\tiny{\textcircled{\tiny\#}}}a
       =a^{\tiny{\textcircled{\tiny\#}}}ab^{\tiny{\textcircled{\tiny\#}}}a
       =a^{\tiny{\textcircled{\tiny\#}}}aa^{\tiny{\textcircled{\tiny\#}}}b
       =a^{\tiny{\textcircled{\tiny\#}}}b$ and
        $aa^{\tiny{\textcircled{\tiny\#}}}
       =ab^{\tiny{\textcircled{\tiny\#}}}aa^{\tiny{\textcircled{\tiny\#}}}
       =ba^{\tiny{\textcircled{\tiny\#}}}aa^{\tiny{\textcircled{\tiny\#}}}
       =ba^{\tiny{\textcircled{\tiny\#}}}.$
\end{proof}

In \cite[Theorem 2.5]{MRT} Malik et al. investigated the reverse order law for two core invertible complex matrices
under the matrix core partial order. By \cite[Theorem 3.1]{XCZ}, we can get that the equations 
$axa=a$ and $xax=x$ in \cite[Theorem 2,14]{RDD} can be dropped.

\begin{lem} \emph{\cite[Theorem 3.1]{XCZ}} \label{five-equations-yy}
Let $a,x \in R$, then $a\in R^{\tiny\textcircled{\tiny\#}}$ with core inverse $x$ if and only if
$(ax)^{\ast}=ax$, $xa^{2}=a$ and $ax^{2}=x$.
\end{lem}

\begin{thm} \label{Reverse-order1}
Let $a,~b \in R^{\tiny{\textcircled{\tiny\#}}}$ with $a\overset{\tiny{\textcircled{\tiny\#}}}\leq b$. Then:
\begin{itemize}
\item[{\rm (1)}] $(ab)^{\tiny{\textcircled{\tiny\#}}}
                 =b^{\tiny{\textcircled{\tiny\#}}}a^{\tiny{\textcircled{\tiny\#}}}
                 =(a^{\tiny{\textcircled{\tiny\#}}})^{2}=(a^{2})^{\tiny{\textcircled{\tiny\#}}}=(ba)^{\tiny{\textcircled{\tiny\#}}}$;
\item[{\rm (2)}] $ab\in R^{EP}$ whenever $a\in R^{EP}$.
\end{itemize}
\end{thm}
\begin{proof}
$(1)$ Suppose that $a\overset{\tiny{\textcircled{\tiny\#}}}\leq b$. Then $a^{\tiny{\textcircled{\tiny\#}}}b=b^{\tiny{\textcircled{\tiny\#}}}a$ by Proposition \ref{a-b-core}. Thus,
$b^{\tiny{\textcircled{\tiny\#}}}a^{\tiny{\textcircled{\tiny\#}}}
 =b^{\tiny{\textcircled{\tiny\#}}}aa^{\tiny{\textcircled{\tiny\#}}}a^{\tiny{\textcircled{\tiny\#}}}
 =a^{\tiny{\textcircled{\tiny\#}}}ba^{\tiny{\textcircled{\tiny\#}}}a^{\tiny{\textcircled{\tiny\#}}}
 =a^{\tiny{\textcircled{\tiny\#}}}aa^{\tiny{\textcircled{\tiny\#}}}a^{\tiny{\textcircled{\tiny\#}}}
 =a^{\tiny{\textcircled{\tiny\#}}}a^{\tiny{\textcircled{\tiny\#}}}
 =(a^{\tiny{\textcircled{\tiny\#}}})^{2}
 =(a^{2})^{\tiny{\textcircled{\tiny\#}}}
 =(ba)^{\tiny{\textcircled{\tiny\#}}}.$
Let $x=b^{\tiny{\textcircled{\tiny\#}}}a^{\tiny{\textcircled{\tiny\#}}}$. Then
\begin{eqnarray*}
&&abx=abb^{\tiny{\textcircled{\tiny\#}}}a^{\tiny{\textcircled{\tiny\#}}}
       =aba^{\tiny{\textcircled{\tiny\#}}}a^{\tiny{\textcircled{\tiny\#}}}
       =aaa^{\tiny{\textcircled{\tiny\#}}}a^{\tiny{\textcircled{\tiny\#}}}
       =aa^{\tiny{\textcircled{\tiny\#}}}
       =(aa^{\tiny{\textcircled{\tiny\#}}})^{*}
       =(abb^{\tiny{\textcircled{\tiny\#}}}a^{\tiny{\textcircled{\tiny\#}}})^{*};\\
&&x(ab)^{2}=b^{\tiny{\textcircled{\tiny\#}}}a^{\tiny{\textcircled{\tiny\#}}}(ab)^{2}
      =b^{\tiny{\textcircled{\tiny\#}}}a^{\tiny{\textcircled{\tiny\#}}}a(ba)b
      = a^{\tiny{\textcircled{\tiny\#}}}a^{\tiny{\textcircled{\tiny\#}}}aa^{2}b
      = a^{\tiny{\textcircled{\tiny\#}}}a^{2}b
      = ab;\\
&&abx^{2}=ab(b^{\tiny{\textcircled{\tiny\#}}}a^{\tiny{\textcircled{\tiny\#}}})^{2}
             =a(ba^{\tiny{\textcircled{\tiny\#}}})a^{\tiny{\textcircled{\tiny\#}}}(a^{\tiny{\textcircled{\tiny\#}}})^{2}
             =(a^{\tiny{\textcircled{\tiny\#}}})^{2}
             = b^{\tiny{\textcircled{\tiny\#}}}a^{\tiny{\textcircled{\tiny\#}}}.
\end{eqnarray*}
Thus $(ab)^{\tiny{\textcircled{\tiny\#}}}=b^{\tiny{\textcircled{\tiny\#}}}a^{\tiny{\textcircled{\tiny\#}}}$ by Lemma \ref{five-equations-yy}.

$(2)$  Suppose that $a\in R^{EP}$. Then $a^{\tiny{\textcircled{\tiny\#}}}a=aa^{\tiny{\textcircled{\tiny\#}}}.$ Thus
\begin{eqnarray*}
&&b^{\tiny{\textcircled{\tiny\#}}}a^{\tiny{\textcircled{\tiny\#}}}ab
                                       =b^{\tiny{\textcircled{\tiny\#}}}aa^{\tiny{\textcircled{\tiny\#}}}b
                                       =a^{\tiny{\textcircled{\tiny\#}}}ba^{\tiny{\textcircled{\tiny\#}}}b
                                       =a^{\tiny{\textcircled{\tiny\#}}}b
                                       =a^{\tiny{\textcircled{\tiny\#}}}a;\\
&&abb^{\tiny{\textcircled{\tiny\#}}}a^{\tiny{\textcircled{\tiny\#}}}
       =abb^{\tiny{\textcircled{\tiny\#}}}a(a^{\tiny{\textcircled{\tiny\#}}})^{2}
       =aba^{\tiny{\textcircled{\tiny\#}}}b(a^{\tiny{\textcircled{\tiny\#}}})^{2}
       =aaa^{\tiny{\textcircled{\tiny\#}}}a(a^{\tiny{\textcircled{\tiny\#}}})^{2}
       =aa^{\tiny{\textcircled{\tiny\#}}}.
\end{eqnarray*}
By $(1)$, we have
$b^{\tiny{\textcircled{\tiny\#}}}a^{\tiny{\textcircled{\tiny\#}}}ab
=(ab)^{\tiny{\textcircled{\tiny\#}}}ab
=ab(ab)^{\tiny{\textcircled{\tiny\#}}}
=abb^{\tiny{\textcircled{\tiny\#}}}a^{\tiny{\textcircled{\tiny\#}}}$,
hence $ab\in R^{EP}.$
\end{proof}

\section{ Relationships between the core partial order and other partial orders }\label{a}

In this section, we consider the relationships between core partial order and other partial orders.
Recall that the left star partial order $a$ $\ast \!\!\leq b$ in $R$ is defined by: $a^{\ast}a=a^{\ast}b$ and $aR\subseteq bR$.
The right sharp partial order $a\leq_{\#} b$ in $R^\#$ is defined by: $aa^\#=ba^\#$ and $Ra\subseteq Rb$.
Let us start with a auxiliary lemma.

\begin{lem} \emph{\cite{BT}} \label{core-star-sharp}
Let $a\in R^{\tiny{\textcircled{\tiny\#}}}$ and $b\in R$. Then $a\overset{\tiny{\textcircled{\tiny\#}}}\leq b$ if and only if
$a$ $\ast\!\!\leq b$ and $a\leq_{\#} b$.
\end{lem}

In \cite[Theorem 4.10]{RD}, Raki\'{c} and Djordjevi\'{c} gave the relationship between the core partial order and the minus partial order
for $a, b\in R^{\tiny{\textcircled{\tiny\#}}}$.
For instance, it is proved that $a\overset{\tiny{\textcircled{\tiny\#}}}\leq b$ if and only if $a\overset{-}\leq b$ and
$b^{\tiny{\textcircled{\tiny\#}}}ab^{\tiny{\textcircled{\tiny\#}}}=a^{\tiny{\textcircled{\tiny\#}}}$.
By Lemma \ref{core-star-sharp}, the core partial order implies the left star partial order
and the right sharp partial order.
Motivated by \cite[Theorem 4.10]{RD}, we have the following theorem.

\begin{thm} \label{core-and-other}
Let $a,~b \in R^{\tiny{\textcircled{\tiny\#}}}$. Then the following are equivalent:
\begin{itemize}
\item[{\rm (1)}] $a\overset{\tiny{\textcircled{\tiny\#}}}\leq b$;
\item[{\rm (2)}] $a$ $\ast\!\!\leq b$  and $ba^{\tiny{\textcircled{\tiny\#}}}b=a$;
\item[{\rm (3)}] $a$ $\ast \!\!\leq b$ and $b^{\tiny{\textcircled{\tiny\#}}}aa^{\tiny{\textcircled{\tiny\#}}}=a^{\tiny{\textcircled{\tiny\#}}}$;
\item[{\rm (4)}] $a$ $\ast \!\!\leq b$ and $b^{\tiny{\textcircled{\tiny\#}}}ab^{\tiny{\textcircled{\tiny\#}}}=a^{\tiny{\textcircled{\tiny\#}}}$;
\item[{\rm (5)}] $a\leq_{\#} b$ and $ba^{\tiny{\textcircled{\tiny\#}}}b=a$;
\item[{\rm (6)}] $a\leq_{\#} b$ and $a^{\tiny{\textcircled{\tiny\#}}}ab^{\tiny{\textcircled{\tiny\#}}}=a^{\tiny{\textcircled{\tiny\#}}}$.
\end{itemize}
\end{thm}
\begin{proof}
$(1) \Rightarrow (2)$-$(6)$ It is obviously by Theorem \ref{core-pr1}, Theorem \ref{core-minus-regular} and Lemma \ref{core-star-sharp}.\\
$(2) \Rightarrow (1)$ Suppose that $a$ $\ast \!\!\leq b$ and $ba^{\tiny{\textcircled{\tiny\#}}}b=a$. Then $a^{*}a=a^{*}b$ and $aR\subseteq bR$.
We have $a^{*}a=a^{*}b$ if and only if $a^{\tiny{\textcircled{\tiny\#}}}a=a^{\tiny{\textcircled{\tiny\#}}}b$ by Lemma \ref{lemma-partial2}, thus $aa^{\tiny{\textcircled{\tiny\#}}}=ba^{\tiny{\textcircled{\tiny\#}}}ba^{\tiny{\textcircled{\tiny\#}}}
=ba^{\tiny{\textcircled{\tiny\#}}}aa^{\tiny{\textcircled{\tiny\#}}}=ba^{\tiny{\textcircled{\tiny\#}}}.$\\
$(3) \Rightarrow (1)$ Suppose that $a$ $\ast\!\!\leq b$. We have $a=bs$ for some $s\in R$, then $a=bs=bb^{\tiny{\textcircled{\tiny\#}}}bs=bb^{\tiny{\textcircled{\tiny\#}}}a$, thus
$aa^{\tiny{\textcircled{\tiny\#}}}=bb^{\tiny{\textcircled{\tiny\#}}}aa^{\tiny{\textcircled{\tiny\#}}}=ba^{\tiny{\textcircled{\tiny\#}}}$.\\
$(4) \Rightarrow (1)$ Suppose $a$ $\ast \!\!\leq b$ and $b^{\tiny{\textcircled{\tiny\#}}}ab^{\tiny{\textcircled{\tiny\#}}}=a^{\tiny{\textcircled{\tiny\#}}}.$
Then $a^{\ast}a=a^{\ast}b$, thus by Lemma \ref{lemma-partial2}, we have $a^{\tiny{\textcircled{\tiny\#}}}a=a^{\tiny{\textcircled{\tiny\#}}}b.$
By $a$ $\ast \!\!\leq b$, we have $a=bb^{\tiny{\textcircled{\tiny\#}}}a$, which gives
$$ba^{\tiny{\textcircled{\tiny\#}}}=b(b^{\tiny{\textcircled{\tiny\#}}}ab^{\tiny{\textcircled{\tiny\#}}})=ab^{\tiny{\textcircled{\tiny\#}}}.$$
Pre-multiplication of $b^{\tiny{\textcircled{\tiny\#}}}ab^{\tiny{\textcircled{\tiny\#}}}=a^{\tiny{\textcircled{\tiny\#}}}$ by $a$ and
post-multiplication of $b^{\tiny{\textcircled{\tiny\#}}}ab^{\tiny{\textcircled{\tiny\#}}}=a^{\tiny{\textcircled{\tiny\#}}}$ by $bb^{\tiny{\textcircled{\tiny\#}}}$ yield
$$aa^{\tiny{\textcircled{\tiny\#}}}bb^{\tiny{\textcircled{\tiny\#}}}
=ab^{\tiny{\textcircled{\tiny\#}}}ab^{\tiny{\textcircled{\tiny\#}}}bb^{\tiny{\textcircled{\tiny\#}}}=aa^{\tiny{\textcircled{\tiny\#}}}.$$
Since $a^{\tiny{\textcircled{\tiny\#}}}a=a^{\tiny{\textcircled{\tiny\#}}}b$, we have
$aa^{\tiny{\textcircled{\tiny\#}}}=aa^{\tiny{\textcircled{\tiny\#}}}bb^{\tiny{\textcircled{\tiny\#}}}
=aa^{\tiny{\textcircled{\tiny\#}}}ab^{\tiny{\textcircled{\tiny\#}}}=ab^{\tiny{\textcircled{\tiny\#}}}$.
Thus by $ba^{\tiny{\textcircled{\tiny\#}}}=ab^{\tiny{\textcircled{\tiny\#}}}$ and the definition of core partial order,
we have $a\overset{\tiny{\textcircled{\tiny\#}}}\leq b$.\\
$(5) \Rightarrow (1)$ Suppose that $a\leq_{\#} b$ and $ba^{\tiny{\textcircled{\tiny\#}}}b=a$. Then $aa^\#=ba^\#$ and $Ra\subseteq Rb$,
by Lemma \ref{lemma-partial2}, we have
$aa^\#=ba^\#$ if and only if $aa^{\tiny{\textcircled{\tiny\#}}}=ba^{\tiny{\textcircled{\tiny\#}}}$, thus
$a^{\tiny{\textcircled{\tiny\#}}}a=a^{\tiny{\textcircled{\tiny\#}}}ba^{\tiny{\textcircled{\tiny\#}}}b=
a^{\tiny{\textcircled{\tiny\#}}}aa^{\tiny{\textcircled{\tiny\#}}}b=a^{\tiny{\textcircled{\tiny\#}}}b.$\\
$(6) \Rightarrow (1)$ By $(5) \Rightarrow (1)$, we only need to prove $a^{\tiny{\textcircled{\tiny\#}}}a=a^{\tiny{\textcircled{\tiny\#}}}b$.\\
Since $Ra\subseteq Rb$ is equivalent to $a=ab^{\tiny{\textcircled{\tiny\#}}}b$, we have
$a^{\tiny{\textcircled{\tiny\#}}}a=a^{\tiny{\textcircled{\tiny\#}}}ab^{\tiny{\textcircled{\tiny\#}}}b=a^{\tiny{\textcircled{\tiny\#}}}b.$
\end{proof}

The right star partial order $a\leq\!\!\ast$ $b$ is defined as: $aa^{\ast}=ba^{\ast}$ and $Ra\subseteq Rb.$

\begin{rem}\label{slip2}
\emph{Let $a\in R^{\tiny{\textcircled{\tiny\#}}}$ and $b\in R^{EP}$.
In \cite[Theorem $2.9$]{M2}, it is claimed that $a\overset{\tiny{\textcircled{\tiny\#}}}\leq b$
if and only if $a\leq\!\!\ast$ $b$ and $b^{\tiny{\textcircled{\tiny\#}}}ab^{\tiny{\textcircled{\tiny\#}}}=a^{\tiny{\textcircled{\tiny\#}}}$
in the complex matrix case. But it is not true.
In fact, let
$A=\left[\begin{matrix}
              1 & 1 \\
              0 & 0
       \end{matrix}
       \right],
B=\left[\begin{matrix}
              1 & 1 \\
              0 & 1
       \end{matrix}
       \right]\in M_{2}\{\mathbb{C}\}$\normalsize,
then $A$ is core invertible, $B$ is an EP matrix and the condition $A\overset{\tiny{\textcircled{\tiny\#}}}\leq B$ is satisfied,
but $AA^{\ast}\neq BA^{\ast}.$}
\end{rem}

The equivalence of $(2)$-$(4)$ in the following proposition for the complex matrices has been proved by Malik et al. in \cite[Lemma 19]{MRT}.
\begin{prop} \label{k-core}
Let $a\in R^{\tiny{\textcircled{\tiny\#}}}, ~b\in R$ with $a\overset{\tiny{\textcircled{\tiny\#}}}\leq b$. Then the following conditions are equivalent:
\begin{itemize}
\item[{\rm (1)}] $a\overset{\#}\leq b$;
\item[{\rm (2)}] $ab=ba$;
\item[{\rm (3)}] $a^{2}\overset{\tiny{\textcircled{\tiny\#}}}\leq b^{2}$;
\item[{\rm (4)}] $a^{k}\overset{\tiny{\textcircled{\tiny\#}}}\leq b^{k}$, for any $k\geq 2$.
\end{itemize}
\end{prop}
\begin{proof}
By Lemma \ref{lemma-partial2}, we have $a\overset{\tiny{\textcircled{\tiny\#}}}\leq b$ if and only if both $a^{*}a=a^{*}b$ and $ba=a^{2}$ hold.

$(1)\Rightarrow(2)$ is obvious by Lemma \ref{lemma-partial1}.

$(2)\Rightarrow(4)$ If $ab=ba$, then $ab=ba=a^{2}$ by Lemma \ref{lemma-partial2}.
If $k\geq 2$, first show $ab^{k-1}=a^{k}$.
 When $k=2$, $ab=ba=a^{2}$;
 when $k>2$, $ab^{k-1}=a^{2}b^{k-2}=a^{2}bb^{k-3}=a^{3}b^{k-3}= \cdots=a^{k}.$
Next prove $(a^{k})^{\tiny{\textcircled{\tiny\#}}}a^{k}=(a^{k})^{\tiny{\textcircled{\tiny\#}}}b^{k}.$
In fact,$(a^{k})^{\tiny{\textcircled{\tiny\#}}}b^{k}
            =(a^{\tiny{\textcircled{\tiny\#}}})^{k}b^{k}
            =(a^{\tiny{\textcircled{\tiny\#}}})^{k-1}a^{\tiny{\textcircled{\tiny\#}}}bb^{k-1}
            =(a^{\tiny{\textcircled{\tiny\#}}})^{k-1}a^{\tiny{\textcircled{\tiny\#}}}ab^{k-1}
            =(a^{\tiny{\textcircled{\tiny\#}}})^{k}ab^{k-1}
            =(a^{k})^{\tiny{\textcircled{\tiny\#}}}ab^{k-1}
            =(a^{k})^{\tiny{\textcircled{\tiny\#}}}a^{k}.$
Similarly, $b^{k}(a^{k})^{\tiny{\textcircled{\tiny\#}}}=a^{k}(a^{k})^{\tiny{\textcircled{\tiny\#}}}.$

$(4)\Rightarrow(3)$ Taking $k=2$.

$(3)\Rightarrow(1)$
       If $a^{2}\overset{\tiny{\textcircled{\tiny\#}}}\leq b^{2}$,
       then $(a^{2})^{\tiny{\textcircled{\tiny\#}}}a^{2}=(a^{2})^{\tiny{\textcircled{\tiny\#}}}b^{2}.$ And
       $$(a^{2})^{\tiny{\textcircled{\tiny\#}}}a^{2}
       =(a^{\tiny{\textcircled{\tiny\#}}})^{2}a^{2}
       =a^{\tiny{\textcircled{\tiny\#}}}a^{\tiny{\textcircled{\tiny\#}}}a^{2}
       =a^{\tiny{\textcircled{\tiny\#}}}a
       =a^{\#}a,$$
       $$(a^{2})^{\tiny{\textcircled{\tiny\#}}}b^{2}
       =(a^{\tiny{\textcircled{\tiny\#}}})^{2}b^{2}
       =a^{\tiny{\textcircled{\tiny\#}}}a^{\tiny{\textcircled{\tiny\#}}}bb
       =a^{\tiny{\textcircled{\tiny\#}}}a^{\tiny{\textcircled{\tiny\#}}}ab
       =a^{\#}b,$$
thus $a^{\#}a=a^{\#}b$. Hence $a^{2}=aaa^{\#}a=aaa^{\#}b=ab=ba$ by $ba=a^{2}$.
\end{proof}

A complex matrix $A$ is called range-Hermite, if $\mathcal{R}(A)=\mathcal{R}(A^{\ast})$.
In \cite[Theorem 7]{BT}, Baksalary and Trenker proved that for complex matrices $A$ and $B$, if
$A$ is an range-Hermite matrix, then
$A\overset{\tiny{\textcircled{\tiny\#}}}\leq B$ if and only if $A\overset{\ast}\leq B$.
In \cite[Theorem 3.3]{M2}, Mailk proved that for complex matrices $A$ and $B$, if
$A$ is an range-Hermite matrix, then $A\overset{\tiny{\textcircled{\tiny\#}}}\leq B$
if and only if $A\overset{\#}\leq B$.
It is easy to check that the following proposition is valid for elements in rings by \cite[Theorem 3.1]{RDD}.

\begin{prop}  \label{EP-core-other}
Let $a\in R^{EP}$ and $b\in R$. Then the following are equivalent:
\begin{itemize}
\item[{\rm (1)}] $a\overset{\tiny{\textcircled{\tiny\#}}}\leq b$;
\item[{\rm (2)}] $a\overset{\#}\leq b$;
\item[{\rm (3)}] $a\overset{\ast}\leq b$.
\end{itemize}
\end{prop}

\vspace{0.2cm} \noindent {\large\bf Acknowledgements}

This research is supported by the National Natural Science Foundation of China (No. 11371089),
The second author is grateful to China Scholarship Council for giving him a purse for his further study in Universidad Polit\'{e}cnica de Valencia, Spain.


\begin{thebibliography}{99}
\bibitem{BT} O.M. Baksalary, G. Trenkler, Core inverse of matrices, Linear Multilinear Algebra 58 (2010), no. 6, 681-697.

\bibitem{D} M.P. Drazin, Natural structures on semigroups with involution, Bull. Amer. Math. Soc. 84 (1978), no. 1, 139-141.

\bibitem{H2} R.E. Hartwig, How to order regular elements ?, Math. Japon. 25 (1980), 1-13.

\bibitem{HS} R.E. Hartwig, K. Spindelb$\ddot{o}$ck, Matrices for which $A^{\ast}$ and $A^{\dagger}$ commute,  Linear Multilinear Algebra 14 (1984), 241-256.

\bibitem{LPT}  L. Lebtahi, P. Patri\'{c}io, N. Thome, The diamond partial order in rings, Linear Multilinear Algebra 62 (2013), no. 3, 386-395.

\bibitem{M} S.K. Mitra, On group inverses and the sharp order, Linear Algebra Appl. 92 (1987), 17-37.

\bibitem{M2} S.B. Malik, Some more properties of core partial order, Appl. Math. Comput. 221 (2013), 192-201.

\bibitem{MRT} S.B. Malik, L. Rueda, N. Thome, Further properties on the core partial order and other matrix partial orders, Linear Multilinear Algebra, 62 (2014), no. 12, 1629-1648.

\bibitem{Rb} D.S. Raki\'{c}, Decomposition of a ring incluced by minus partial order, Electron. J. Linear Algebra 23 (2012), 1040-1059.

\bibitem{R} D.S. Raki\'{c}, Generalization of sharp and core partial order using annihilators, Banach J. Math. Anal. 9 (2015), no. 3, 228-242.

\bibitem{RD} D.S. Raki\'{c}, D.S. Djordjevi\'{c}, Star, sharp, core and dual core partial order in rings with involution, Appl. Math. Comput. 259 (2015), 800-818.

\bibitem{RDD} D.S. Raki\'{c}, N. \v{C}. Din\v{c}i\'{c}, D.S. Djordjevi\'{c},  Group, Moore-Penrose, core and dual core inverse in rings with involution, Linear Algebra Appl. 463 (2014), 115-133.

\bibitem{WL} H.X. Wang, X.J. Liu, Characterizations of the core inverse and the partial ordering, Linear Multilinear Algebra 63 (2015), no. 9, 1829-1836.

\bibitem{XCZ} S.Z. Xu, J.L. Chen, X.X. Zhang, New characterizations for core inverses in rings with involution, Front. Math. China. 12(1) (2017), 231-246.
\end{thebibliography}
\end{document}